\documentclass[12pt]{amsart}

\usepackage{latexsym,amsmath,amssymb,amsthm,amsfonts}

\usepackage[numeric]{amsrefs}

\usepackage[capitalise]{cleveref}

\usepackage{multido,pst-plot,pstricks,pst-node}
\newcommand{\eps}{\varepsilon}

\newcommand{\E}{{\mathbb{E}}}

\newcommand{\U}{{\mathcal{U}}}
\renewcommand{\P}{{\mathcal{P}}}

\newcommand{\B}{{\mathcal{B}}}
\newcommand{\EE}{{\mathcal{E}}}
\newcommand{\LL}{{\mathcal{L}}}

\newcommand{\p}{{\overline{p}}}

\newcommand{\conv}{{\rm conv}}

\newcommand{\intr}{{\rm int}}

\renewcommand{\phi}{{\boldsymbol{\varphi}}}

\newtheorem{theorem}{Theorem}[section]
\newtheorem{lemma}[theorem]{Lemma}
\newtheorem{proposition}[theorem]{Proposition}

\newtheorem{corollary}[theorem]{Corollary}

\theoremstyle{remark}
\newtheorem{remark}[theorem]{Remark}

\title{A new bound for Hadwiger's covering problem in $\mathbb{E}^3$}

\author{A.\ Prymak}
\address{Department of Mathematics, University of Manitoba, Winnipeg, MB, R3T 2N2, Canada}
\email{prymak@gmail.com}
\thanks{Corresponding author: Andriy Prymak; email: {\tt prymak@gmail.com} \\ \phantom{fff} The author was supported by NSERC of Canada Discovery Grant RGPIN-2020-05357.}

\keywords{Illumination problem, illumination number, covering number, covering by smaller homothetic copies, convex body, linear programming}

\subjclass[2010]{Primary 52A15; Secondary 52A37, 52A40, 52C17, 90C05}

\begin{document}

\begin{abstract}
We show that every $3$-dimensional convex body can be covered by $14$ smaller homothetic copies. The previous result was $16$ established by Papadoperakis in 1999, while a conjecture by Hadwiger is $8$. We 
develop a discretization technique showing that it suffices to verify feasibility of a number of linear programs with rational coefficients, which is done with computer assistance using exact arithmetic.
\end{abstract}

\maketitle

\section{Introduction}

A convex body in the $n$-dimensional Euclidean space $\E^n$ is a convex compact set having non-empty interior. Let $C(A,B)$ be the smallest number of translates of $B\subset\E^n$ required to cover $A\subset\E^n$, and let $\intr(A)$ denote the interior of $A$. Hadwiger~\cite{Ha} asked what is the value $H_n=\min C(K,\intr(K))$, where the minimum is taken over all convex bodies $K$ in $\E^n$. In other words, what is the least number of smaller homothetic copies of an arbitrary convex body $K$ in $\E^n$ needed to cover $K$? It was shown by Boltyanski~\cite{Bo} that this question is equivalent to finding the smallest number of external light sources required to illuminate the boundary of every convex body. It is immediate that $H_n\ge 2^n$ by taking $K$ as an $n$-dimensional cube. A well-known conjecture, which is commonly referred to as the (Levi-)Hadwiger conjecture or as the Gohberg-Markus covering conjecture, is that $H_n=2^n$, but this is known (see Levi~\cite{Le}) only for $n=2$. Below we give a brief overview of the known results about $H_n$. For a detailed history of the question and the survey including many partial results for special classes of convex bodies see, e.g., \cite{Be}. 

The best known explicit upper bound on $H_n$ in high dimensions is a combination of Rogers's result~\cite{Ro} on the covering density of $\E^n$ by translates of arbitrary convex body with the Rogers-Shephard inequality~\cite{Ro-Sh} bounding the volume of the difference body. An interested reader is referred to~\cite{Be}*{Section~2.2} for further details and related results. Here we only state the actual bound, which is
\begin{equation}\label{eqn:asymptotic}
	H_n\le \binom{2n}{n}n(\ln n+\ln\ln n+5),
\end{equation}
where $5$ can be replaced by $4$ for sufficiently large $n$. (Note that the asymptotic behavior of $\binom{2n}{n}$ is $\frac{4^n}{\sqrt{2\pi n}}$.) This was the best asymptotic estimate on $H_n$ for a long time until Huang, Slomka, Tkocz and Vritsiou~\cite{HSTV} recently obtained a remarkable asymptotic improvement of~\eqref{eqn:asymptotic} by showing $H_n\le \exp(-c\sqrt{n}) 4^n$ using ``thin-shell'' volume estimates, here and below $c>0$ is an implicit universal constant. With the aid of the recent results on the Bourgain slicing problem, this was further improved by Campos, van Hintum, Morris and Tiba~\cite{CHMT} leading to the currently best known asymptotic bound on $H_n$:
\begin{equation*}
H_n\le \exp\Bigl(\frac{-cn}{\ln^8 n}\Bigr)4^n.
\end{equation*}

For low dimensions, Lassak~\cite{La} showed that
\begin{equation}\label{eqn:lassak}
	H_n\le (n+1)n^{n-1}-(n-1)(n-2)^{n-1},
\end{equation}
which is better than~\eqref{eqn:asymptotic} for $n\le 5$. For $n=3$ \eqref{eqn:lassak} is $H_3\le 34$, which was improved to $H_3\le20$ by Lassak~\cite{La2}, and then to $H_3\le 16$ by Papadoperakis~\cite{Pa}. Generalizing the technique of~\cite{Pa} to the higher dimensions, in the joint work with Shepelska~\cite{Pr-Sh}, we showed that $H_4\le 96$, $H_5\le 1091$ and $H_6\le 15373$ improving both~\eqref{eqn:asymptotic} and~\eqref{eqn:lassak} for $n=4,5,6$. Diao~\cite{D} brought the last two inequalities down to $H_5\le 1002$ and $H_6\le 14140$.

The main result of this work is that $H_3\le 14$.

\begin{theorem}
	\label{thm:main}
	For any convex body $K\subset\E^3$, $C(K,\intr(K))\le 14$.
\end{theorem}


 Use of computer for the covering conjecture has been suggested in~\cite{Zong}, although that approach does not appear to be feasible with current computational power. Recently computer assistance was used in~\cite{BPR} in a different way to confirm the covering conjecture for the class of convex bodies of constant width in dimensions $6\le n\le 15$.
 
 The discretization method we offer here is new and may be of independent interest for other affine-invariant problems involving three-dimensional convex bodies. More precisely, we construct a finite family of convex polytopes with relatively simple structure and the property that any three-dimensional convex body contains an affine copy of one of these polytopes as well as is contained in a certain related parallelepiped. 

In \cref{sec:discr}, we begin with the description of the approach from~\cite{Pa} which is the starting point of our proof. Next we introduce a certain family of polytopes which are larger than the boxes considered in~\cite{Pa} allowing more efficient covering (\cref{lem:pap-strong}). We proceed with a certain geometric argument (\cref{lem:inside}) to transition from covering by open sets to covering by closed sets. This transition is one of the ingredients needed to enable the use of linear programming to verify the existence of the translates necessary for covering, the other one being the one-dimensional structure of the edges of the related parallelepiped. We conclude with \cref{sec:comp}, where various techniques making the required computations feasible are described.

\section{Discretization and reduction to linear programming}\label{sec:discr}

For a set $A\subset \E^n$ and a family $\B$ of subsets of $\E^n$, we let $C(A,\B)$ be the smallest number of translates of elements of $\B$ required to cover $A$.

Everywhere below ``the cube'' is the unit cube $[0,1]^3$, the faces of the cube are the $2$-dimensional faces $F_{ij}=\{(x_1,x_2,x_3)\in [0,1]^3:x_i=j\}$, $i=1,2,3$, $j=0,1$. Let $E$ denote the $1$-skeleton of the cube, which is the union of all edges and can also be viewed as the union of all relative boundaries of the faces $F_{ij}$.

A very important role in the proof will be played by certain configurations of points from the faces of the cube. Namely, we select a point on each face so that the segment joining the points in each pair of the opposite faces is always perpendicular to these faces. It is convenient to introduce a notation to describe such configurations. For $p\in[0,1]^6$, define
\[
A_p=\begin{pmatrix}
	0 & 1 & p_3 & p_3 & p_5 & p_5 \\
	p_1 & p_1 & 0 & 1 & p_6 & p_6 \\
	p_2 & p_2 & p_4 & p_4 & 0 & 1
\end{pmatrix}.
\]
Now if $e_i$ denotes the $i$-the basic unit vector, then the configurations mentioned above are $V_p:=A_p(\{e_1,\dots,e_6\})$. We also need to work with closed convex hulls of these configurations, so we set $O_p:=A_p(S)$, where $S=\{(\lambda_1,\dots,\lambda_6):\lambda_i\ge0,\sum\lambda_i=1\}$.

Now let us describe the technique of~\cite{Pa}. For arbitrary convex body $K$ in $\E^3$, consider the parallelotope $C$ of minimal volume containing $K$. Without loss of generality (using an affine transform), we can assume that $C=[0,1]^3$. Clearly, each $F_{ij}\cap K$ is non-empty, moreover, by~\cite{Pa}*{Lemma~3(a)} we can choose $q_{ij}\in F_{ij}\cap K$ so that $q_{i0}+e_i=q_{i1}$ for every $i$. In our notations, this means that the six points $\{q_{ij}, i=1,2,3, j=0,1\}$ form a configuration described above, i.e., are exactly the column vectors of $A_p$ for some $p\in[0,1]^6$. Let $Q=\{q_{ij}, i=1,2,3, j=0,1\}=V_p$, then $O_p=\conv(Q)\subset K$. Let $\P$ be the family of rectangular parallelotops in $\E^3$ whose edges are parallel to coordinate axes and the sum of the three dimensions (lengths of edges adjacent to a vertex) is strictly less than one. By~\cite{Pa}*{Lemma~3(b)}, for any $P\in\P$ there is a translate of $P$ which is a subset of $\intr(K)$ (in fact, a subset of $\intr(O_p)$ for any $p\in[0,1]^6$). The key argument of~\cite{Pa} is the following lemma. Recall that $E$ is the $1$-skeleton of the cube.
\begin{lemma}[\cite{Pa}*{Lemma~4}]\label{lem:pap}
For any convex body $K$ in $\E^3$
	\[
	C(K,\intr(K)) \le \max\{ C(E\cup V_p,\P): p\in[0,1]^6\}.
	\]
\end{lemma}
The family of covering problems $C(E\cup V_p,\P)$ is $6$-parametric which is more tangible than $C(K,\intr(K))$ involving a general convex body in $\E^3$.

The proof of~\cite{Pa}*{Lemma~4} actually shows that
\[
C(K,\intr(K)) \le \max\{ C(E\cup V_p,\intr(K)): p\in[0,1]^6\},
\]
so by $O_p\subset K$, $p\in[0,1]^6$, we immediately obtain the following stronger lemma.
\begin{lemma}\label{lem:pap-strong}
	For any convex body $K$ in $\E^3$
	\[
	C(K,\intr(K)) \le \max\{ C(E\cup V_p,\intr(O_p)): p\in[0,1]^6\}.
	\]
\end{lemma}
Observe that (by the proof of~\cite{Pa}*{Lemma~3(b)}) any polytope $O_p$, $p\in[0,1]^6$, contains a translate of any parallelotope $P\in\P$. Thus, \cref{lem:pap-strong} uses larger sets for covering, and \emph{this is precisely what allows to obtain the main result of this work and an improvement over~\cite{Pa}}. On the other hand, the polytopes $O_p$, $p\in[0,1]^6$, depend on specific $p$ and have more complicated structure than the universal (independent of $p$) family of the parallelotopes $\P$, so we will employ computer assistance to estimate $\max\{ C(E\cup V_p,\intr(O_p)): p\in[0,1]^6\}$. 

\begin{remark}
	It is not hard to see that $14$ is the best one can do with the suggested approach. Indeed, $\max\{ C(E\cup V_p,\intr(O_p)): p\in[0,1]^6\}\ge C(E\cup V_\p,\intr(O_\p))$, where $\p=(\frac12,\dots,\frac12)$. Consider the $14$-element set consisting of the $8$ vertices of the cube and the points $V_\p$. No translate of $\intr(O_\p)$ can cover two points in this set. Therefore, $C(E\cup V_\p,\intr(O_\p))\ge14$.
\end{remark}

To bound $\max\{ C(E\cup V_p,\intr(O_p)): p\in[0,1]^6\}$, we will partition the configuration space $[0,1]^6$ into smaller parallelepipeds $P$ of the form $P=\prod_i [a_i,a_i+\delta_i]$, where $0\le a_i<a_i+\delta_i\le 1$, $1\le i\le 6$. For each $P$, we find an independent of $p\in P$ polytope $Q_P$ that can be used for covering. Namely, we define $Q_P=\cap_{v\in U_P}O_v$ where $U_P$ is the set of all $64$ vertices of $P$. We need the following properties relating $Q_P$ and $O_p$ for $p\in P$.

\begin{lemma}\label{lem:inside}
	For any $p \in P$:\\ (i) $Q_P\subset O_p$;\\ (ii) there is a translate of $Q_P$ which is a subset of the interior of $O_p$.
\end{lemma}
\begin{proof}
	(i) We need the following fact. {\it Suppose $[q,r]$ is a segment parallel to one of the coordinate axes such that $p\in [q,r]\subset P$ for some point $p$. Then $O_q\cap O_r \subset O_p$.} Without loss of generality, we can assume that $q=(q_1,p_2,\dots,p_6)$, $r=(r_1,p_2,\dots,p_6)$ and $q_1\le p_1\le r_1$. For a point $x\in O_q\cap O_r$ there exist $\lambda,\mu\in S$ such that $x=A_q\lambda=A_r\mu$. Define $f(t):=A_p(t\lambda+(1-t)\mu)$, $t\in[0,1]$, clearly $f(t)\in O_p$ for any $t\in[0,1]$. Moreover, it is immediate that the first and the third coordinates of $f(t)$ are equal to $x_1$ and $x_3$, respectively. Now the second coordinate of $f(t)$ equals
	\[
	g(t):=t((\lambda_1+\lambda_2)p_1+\lambda_4+(\lambda_5+\lambda_6)p_6)+(1-t)((\mu_1+\mu_2)p_1+\mu_4+(\mu_5+\mu_6)p_6).
	\]
	Observe that by  $x=A_q\lambda=A_r\mu$ we have 
	\[
	(g(0)-x_2)(g(1)-x_2)=(\mu_1+\mu_2)(p_1-r_1)(\lambda_1+\lambda_2)(p_1-q_1)\le0,
	\]
	so, by continuity, $x=f(t)$ for some $t\in[0,1]$ proving the desired fact.
	
	The proof of (i) is completed by iterative application of the established fact. Namely, for arbitrary $p\in P$, the first step is
	\begin{equation}\label{eqn:1st inclusion}
	O_{(a_1,p_2,\dots,p_6)}\cap O_{(a_1+\delta_1,p_2,\dots,p_6)} \subset O_p.
	\end{equation}
	The second step is application of the fact along the second coordinate for each of the points in the left hand side of~\eqref{eqn:1st inclusion}:
	\begin{align*}
	O_{(a_1,a_2,p_3,\dots,p_6)}\cap O_{(a_1,a_2+\delta_2,p_3,\dots,p_6)} &\subset O_{(a_1,p_2,\dots,p_6)}, \\
	O_{(a_1+\delta_1,a_2,p_3,\dots,p_6)}\cap O_{(a_1+\delta_1,a_2+\delta_2,p_3,\dots,p_6)} &\subset O_{(a_1+\delta_1,p_2,\dots,p_6)},
	\end{align*}
	which, in combination with~\eqref{eqn:1st inclusion}, gives
	\[
	O_{(a_1,a_2,p_3,\dots,p_6)}\cap O_{(a_1,a_2+\delta_2,p_3,\dots,p_6)} \cap
	O_{(a_1+\delta_1,a_2,p_3,\dots,p_6)}\cap O_{(a_1+\delta_1,a_2+\delta_2,p_3,\dots,p_6)} \subset O_p.
	\]
	Continuing in this manner, we arrive at $Q_P\subset O_p$ after the sixth step.

	(ii) First we remark that $Q_P$ has no common points with the boundary of the cube. Indeed, observe that if none of $q_i$ is zero or one, then $O_q$ intersects the boundary of the cube only at the points of the set $V_q$ (recall that $V_q=A_q(\{e_1,\dots,e_6\})$). For any given point $x$ on the boundary of the cube, since all $\delta_i>0$, it is easy to choose $q\in \intr(P)$ so that $x\not\in V_q$, which proves the remark.
	
	Without loss of generality, we can assume that $p\in\prod_i [a_i,a_i+\delta_i/2]$. Next we show that for any $x\in Q_P$ there exists $\eps>0$ (possibly depending on $x$) such that $x-(\eps,0,0)\in O_p$. First consider the case $(x_2,x_3)\ne(p_1,p_2)$. We have $r=p+(0,0,\delta,0,\delta,0)\in P$ for any fixed $0<\delta<\delta_i/2$, $i=3,5$, so by (i), we can find $\lambda\in S$ such that $x=A_r\lambda$. Observe that $A_r\lambda-A_p\lambda=(\delta\lambda',0,0)$ where $\lambda':=\lambda_3+\dots+\lambda_6$. The set $A_r(\{\mu\in S:\mu_1+\mu_2=1\})$ is the segment joining $(0,p_1,p_2)$ with $(1,p_1,p_2)$ which does not contain $x$ by $(x_2,x_3)\ne(p_1,p_2)$. Therefore, we can assume that $\lambda_1+\lambda_2<1$, i.e. $\lambda'>0$. We have $O_p\ni A_p\lambda=x-(\delta\lambda',0,0)$ as required. Now consider the case $(x_2,x_3)=(p_1,p_2)$. Since $x\in Q_P$ which has no common points with the boundary of the cube, we have $x_1>0$. In addition, $x\in Q_P\subset O_p$ and $(0,p_1,p_2)=A_pe_1\in O_p$, so the segment $\{(t,p_1,p_2):0\le t\le x_1\}$ is a subset of $O_p$ which yields the required since $x_1>0$.
	
	Now it remains to show that the selection of $\eps>0$ as above can be made independent of $x\in Q_P$. Note that $Q_P$ is a closed convex polytope as an intersection of finitely many closed convex polytopes $O_v$, $v\in U_P$. Define $h:Q_P\to(0,\infty)$ by $h(x):=\max\{\eps>0:x-(\eps,0,0)\in O_p\}$. We need to show that $\inf_{x\in Q_P}h(x)>0$. It is easy to observe that the convexity of $O_p$ and $Q_P$ implies the concavity of $h$: if $x,y\in Q_P$ and $t\in[0,1]$, then $x-(h(x),0,0),y-(h(y),0,0)\in O_p$, so $O_p\ni t(x-(h(x),0,0))+(1-t)(y-(h(y),0,0))=(tx+(1-t)y)-(th(x)+(1-t)h(y),0,0)$ and $h(tx+(1-t)y)\ge th(x)+(1-t)h(y)$. Since $Q_P$ is a closed convex polytope, any point $x\in Q_P$ can be written as a convex combination of the (finitely many) vertices of $Q_P$. By the concavity of $h$ this implies $h(x)\ge \min_{w\in W_P} h(w)>0$, where $W_P$ is the set of the vertices of $Q_P$.
\end{proof}

Let $R_P=\cup_{p\in P} V_p$, i.e. the union of six rectangles on the facets of $[0,1]^3$ where the points $V_p$ vary as $p$ varies over $P$. The following is immediate by \cref{lem:inside}.

\begin{corollary}\label{cor:box}
		We have
		\[
		\max\{ C(E\cup V_p,\intr(O_p)): p\in P\} \le C(E\cup R_P, Q_P). 
		\]
\end{corollary}

 With the goal of obtaining $C(E\cup R_P, Q_P)\le 14$, we need to introduce certain structure of such covers and describe $14$ translates of $Q_P$ that will be used. For each vertex of the cube, there will be a translate of $Q_P$ covering that vertex and certain parts of each of the three adjacent edges. Each of the remaining $14-8=6$ translates covers one of the rectanlges $R_P\cap F_{ij}$ together with a ``middle'' portion of one of the edges. Observe that with this covering structure each edge from some $6$ edges would be covered by the two translates of $Q_P$ corresponding to the vertices which are the endpoints of the edge, while each edge from the other $6$ edges would be covered by three translates of $Q_P$: two corresponding to the vertices, and one covering the uncovered ``middle'' portion of the edge which also covers one of the rectangles $R_P\cap F_{ij}$. Since $Q_P$ is a \emph{closed} convex polytope and edges are one-dimensional, once the structure of the covering is fixed, it is possible to decide if such a cover exists by verifying feasibility of the corresponding system of linear inequalities which will be described below.
 
 Let $E_i=\{(v_{i1},v_{i2},v_{i3})+tu_i:t\in[0,1]\}$, $i=1,\dots,12$, be the $12$ edges of the cube, where $(v_{i1},v_{i2},v_{i3})$ is a vertex of the cube and $u_i\in\{e_1,e_2,e_3\}$. Consider any injective mapping $\tau:\{1,2,3\}\times\{0,1\}\to\{0,\dots,11\}$ which indicates for each face $F_{ij}$ middle portion of which edge will be covered together with $R_P\cap F_{ij}$. Let $\EE=\tau(\{1,2,3\}\times\{0,1\})$.
 
 For each vertex $v$ of the cube, introduce a translate vector $a(v)$. Using the half-space representation of the convex polytope $Q_P$, the condition $v\in a(v)+Q_P$ can be written as a system of linear inequalities containing the components of the vector $a$. Combining these over the vertices, we have a system of $8q$ linear inequalities on $24=8\cdot 3$ variables (three per each vertex), where $q$ is the number of facets of $Q_P$. Next, for each face $F_{ij}$, we introduce a translate $b_{ij}$ and require that $R_P\cap F_{ij}\subset b_{ij}+Q_P$. Since $R_P\cap F_{ij}$ is a rectangle and $Q_P$ is convex, it suffices to verify that each vertex of $R_P\cap F_{ij}$ belongs to $b_{ij}+Q_P$. We now have $24+6\cdot 3=42$ variables and a number of linear constraints ensuring that $v\in a(v)+Q_P$ for each vertex $v$ and that $R_P\cap F_{ij}\subset b_{ij}+Q_P$ for each face $F_{ij}$. It remains to cover the interiors of the edges $E_k$. Let $T$ be the image of $\tau$; it consists of $6$ elements. First we consider the easier case $k\not\in T$. We will cover $E_k=[v_k,v_k+u_k]$ using the translates covering the endpoints of $E_k$. We introduce a variable $t_k$ and require that $v_k+t_ku_k\in a(v_k)+Q_P$ and $v_k+t_ku_k\in a(v_k+u_k)+Q_P$ at the same time for some value of $t_k$. In this case, the point $v_k+t_ku_k$ belongs to both convex polytopes $a(v_k)+Q_P$ and $a(v_k+u_k)+Q_P$ which contain $v_k$ and $v_k+u_k$, respectively, thus $E_k=[v_k,v_k+u_k]$ is completely covered. For the case $k\in T$, we proceed similarly with the difference that two intermediary points are needed instead of one. Namely, when $k\in T$, we introduce two variables $t_k$, $s_k$ and require $v_k+t_ku_k\in a(v_k)+Q_P$, $v_k+s_ku_k\in a(v_k+u_k)+Q_P$ and $v_k+t_ku_k,v_k+s_ku_k\in b_{ij}+Q_P$, where $k=\tau(i,j)$. Altogether, we have $42+6+6\cdot2=60$ variables and $(8+4\cdot6+2\cdot 6+4\cdot 6)q=68q$ constraints. Let us denote this system of linear inequalities as $\LL(P,\tau)$ (observe that $Q_P$ depends on $P$). We have just proved the following.
 
\begin{proposition}\label{prop:main}
	If there exists an injective mapping $\tau:\{1,2,3\}\times\{0,1\}\to\{1,\dots,12\}$ such that $\LL(P,\tau)$ has a solution, then
	$
	C(E\cup R_P, Q_P)\le 14
	$.
\end{proposition}

\section{Computer verification}\label{sec:comp}

With \cref{lem:pap-strong,cor:box,prop:main} at hand, it is clear how the task of proving $C(K,\intr(K))\le14$ can be discretized and performed on a computer in a finite number of steps, provided the described covering structure works. Indeed, for a large positive integer $M$, consider the partition $\U_M$ of $[0,1]^6$ into $M^6$ congruent cubes with side length $M^{-1}$, and find a suitable $\tau$ (there are finitely many possibilities) for each cube $P\in\U_M$ so that $\LL(P,\tau)$ has a solution. Since all the coordinates of the vertices of $P$ are rational, it is easily seen that all the coefficients of the system $\LL(P,\tau)$ are rational. So, the required verification can be performed using exact computations. Such a straightforward approach will not be feasible in practice, and we will apply a number of ideas to shorten the required computation time. We remark that in our discretization all the inequalities in the systems $\LL(P,\tau)$ are non-strict, which is crucial for applicability of linear programming techniques. Although \cref{lem:pap-strong} calls for covers using the interiors of certain polytopes, the transition from open sets to the closed sets was achieved in \cref{lem:inside}~(ii).

First let us utilize the available symmetries. The collection of $6$-tuples of points $\{V_p:p\in[0,1]^6\}$ from $\E^3$ (in fact, from the boundary of the cube) is invariant under symmetries of the cube. These symmetries can be obtained as compositions of the following two transformations: interchange of any two coordinates $x_i \longleftrightarrow x_j$ and mappings of the type $x_i \longleftrightarrow 1-x_i$. Using the latter, we can assume $p_1,p_2,p_3\in[0,\frac12]$. Then using the former, we interchange the order of the rows in the matrix $A_p$ to assume that $p_2+p_4\le p_1+p_6\le p_3+p_5$. Denote $D:=\{p\in[0,\tfrac12]^3\times[0,1]^3:p_2+p_4\le p_1+p_6\le p_3+p_5\}$. Due to the symmetries described above, we have
\begin{equation}\label{eqn:sym-cover}
\max\{ C(E\cup V_p,\intr(O_p)): p\in[0,1]^6\}=\max\{ C(E\cup V_p,\intr(O_p)): p\in D\}.
\end{equation}
Next we cover $D$ with certain cubes from $\U_{2M}$. With $h:=\tfrac{1}{2M}$, define
\begin{align*}
	\U_{2M}^D:=\bigcup \Bigl\{\prod_{i=1}^6 [k_ih,(k_i+1)h]:\ &0\le k_1,k_2,k_3\le M-1,\ 0\le k_4,k_5,k_6\le 2M-1, \\
	& k_2+k_4\le k_1+k_6+1\le k_3+k_5+2  \Bigr\}. 
\end{align*}
We claim that $D\subset \U_{2M}^D$. Indeed, given $p\in D$, define $k_i:=\lfloor p_i/h\rfloor$. If $p_i=\tfrac12$ for $i=1,2,3$, modify $k_i$ by setting $k_i:=M-1$, and if $p_i=1$ for $i=4,5,6$, modify $k_i$ by setting $k_i:=2M-1$. Then $k_ih\le p_i\le (k_i+1)h$ (in particular, $p$ belongs to the cube $\prod_{i=1}^6 [k_ih,(k_i+1)h]$), while $p_i=(k_i+1)h$ only if $p_i=\tfrac12$ for $i=1,2,3$ or  $p_i=1$ for $i=4,5,6$. Let us now establish that $k_2+k_4\le k_1+k_6+1$, the inequality $ k_1+k_6\le k_3+k_5+1$ is completely similar. If $k_1=M-1$ and $k_6=2M-1$, then clearly $k_2+k_4\le (M-1)+(2M-1)= k_1+k_6$. Otherwise, we have the strict inequality $p_1+p_6<(k_1+1)h+(k_6+1)h$, which, combined with $k_2h+k_4h\le p_2+p_4\le p_1+p_6$, proves the desired $k_2+k_4\le k_1+k_6+1$.

We use computer assistance to prove the following proposition which implies \cref{thm:main} due to~\eqref{eqn:sym-cover} and $D\subset \U_{20}^D$.

\begin{proposition}
	There exists a collection $\{(P_j,\tau_j)\}_{j=1}^{4660035}$ of pairs $(P_j,\tau_j)$ where $P_j\subset [0,1]^6$ is a rectangular parallelepiped and $\tau_j:\{1,2,3\}\times\{0,1\}\to\{1,\dots,12\}$ is an injective mapping such that $\cup_j P_j=\U_{20}^D$ and $\LL(P_j,\tau_j)$ has a solution for every $j$.
\end{proposition}

\begin{proof}
	We begin with $1882010$ boxes from $\U_{20}^D$ and attempt to find an appropriate $\tau_j$ for each box. There are $1496$  possible choices for each mapping $\tau_j$ if we restrict ourselves to selecting $\tau_j(F)$ as one of the four edges contained in a face $F$ (keeping in mind that the resulting mapping should be injective, i.e. no edge should be selected twice). By working with randomly chosen boxes, we have selected a set of $16$ possible choices for $\tau_j$ one of which worked often. Let us call these $16$ possibilities a short list. 
	
	The algorithm constructing the required collection of boxes and mappings proceeds as follows. First we split the box $[0,\tfrac12]^3\times[0,1]^3$ containing $D$ into $512$ smaller boxes \[
	\prod_{i=1}^6[\tfrac{j_i}4,\tfrac{j_i}4+\tfrac14],
	\quad j_1,j_2,j_3\in\{0,1\}, \quad j_4,j_5,j_6\in\{0,1,2,3\}, 
	\] which we will refer to as regions. Now fix a region. For each starting box $P_j$ from $\U_{20}^D$ in the region (some regions may not contain any boxes $P_j$ from $\U_{20}^D$), test if one of the mappings in the short list works, i.e. $\LL(P_j,\tau_j)$ has a solution. More precisely, we generate $Q_P$ using rational numbers (all the vertices of $Q_P$ are from $\mathbb{Q}^3$) and exact arithmetic implying that all the coefficients in the problem $\LL(P_j,\tau_j)$ are rational. However, we first check the feasibility using a solver working with floating point numbers, which is faster than using exact arithmetic with rational numbers. In case the solver returns the ``not feasible'' result, we assume so, but if the solver returns ``feasible'' with floating point arithmetic, we double check the result using exact computations with rational numbers. Returning to the algorithm, if $\LL(P_j,\tau_j)$ is feasible for some $\tau_j$ from the short list, we are done. Otherwise, we consecutively check feasibility of $\LL(P_j,\tau_j)$ picking $\tau_j$ from the (``long'') list of all $1496$ possibilities for the mapping $\tau_j$. If a successful $\tau_j$ is found, we move it to the beginning of the long list. This approach maintains more useful mappings in the start of the list. As before, if we found a mapping $\tau_j$ such that $\LL(P_j,\tau_j)$ is feasible, we are done. Otherwise, we divide $P_j$ into two smaller boxes along one of the largest dimensions (e.g. $[0,\tfrac1{20}]\times[0,\tfrac1{20}]^5$ could be split as $[0,\tfrac1{40}]\times[0,\tfrac1{20}]^5$ and $[\tfrac1{40},\tfrac1{20}]\times[0,\tfrac1{20}]^5$) and invoke the algorithm recursively. Thus the relation $\cup_j P_j=\U_{20}^D$ is always true, and if the algorithm terminates, then all obtained pairs $(P_j,\tau_j)$ are such that $\LL(P_j,\tau_j)$ has a solution.
	
	 We supplied the SageMath scripts used and the results at~\cite{github}. We have listed the resulting boxes and the mappings for each region in the ``output'' folder.
\end{proof}
	
	We finish with a few remarks about the computations.
	The smallest boxes that needed to be used were obtained after $7$ subdivisions, i.e. have the dimensions $\tfrac1{80}\times(\tfrac1{40})^5$. As the pairs $(P_j,\tau_j)$ are already found by the above algorithm, it suffices to run the verification of feasibility of the corresponding linear programs to check the result. Each box requires approximately two seconds running time, but due to independence from other boxes, this verification can be efficiently parallelized. It is possible to complete the verification using 10 threads on a modern personal computer (i7) in about two weeks. 


{\bf Acknowledgment.} The author is grateful to the anonymous referees for the useful comments.

\end{document}